\documentclass[12pt,a4paper]{amsart}

\usepackage{mathptmx}
\usepackage{mathrsfs}

\usepackage{verbatim}
\usepackage{url}
\usepackage[all]{xy}
\usepackage{color}

\usepackage[colorlinks=true,citecolor=blue]{hyperref}

\usepackage[top=40mm, bottom=35mm, left=35mm, right=35mm]{geometry}

\usepackage{stmaryrd}
\usepackage{epsfig}
\usepackage{amsmath}
\usepackage{amssymb}

\usepackage{amscd}
\usepackage{graphicx}
\usepackage{pstricks}

\theoremstyle{plain}
\newtheorem{thm}{Theorem}[section]
\newtheorem{lem}[thm]{Lemma}

\newtheorem{prop}[thm]{Proposition}
\newtheorem{ques}[thm]{Question}
\newtheorem{cor}[thm]{Corollary}
\theoremstyle{definition}
\newtheorem{de}[thm]{Definition}

\theoremstyle{remark}
\newtheorem{rem}[thm]{Remark}

\def \N {\mathbb N}

\def \Z {\mathbb Z}

\def \F {\mathcal F}

\def \B {\mathcal B}

\def \P {\mathcal P}

\def \id {{\rm id}}

\def \ep {\epsilon}
\def \d {\delta}

\begin{document}

\title[An answer to Furstenberg's problem]{An answer to Furstenberg's problem on topological disjointness}

\author{Wen Huang}
\author{Song Shao}
\author{Xiangdong Ye}

\address{Wu Wen-Tsun Key Laboratory of Mathematics, USTC, Chinese Academy of Sciences and
Department of Mathematics, University of Science and Technology of China,
Hefei, Anhui, 230026, P.R. China.}

\email{wenh@mail.ustc.edu.cn}
\email{songshao@ustc.edu.cn}
\email{yexd@ustc.edu.cn}

\subjclass[2010]{Primary: 37B05; 54H20}

\thanks{This research is supported by NNSF of China (11571335, 11431012, 11371339) and by ``the Fundamental Research Funds for the Central Universities''.}

\date{}

\begin{abstract}
In this paper we give an answer to Furstenberg's problem on topological disjointness. Namely, we show that a
transitive system $(X,T)$ is disjoint from all minimal systems if and only if
$(X,T)$ is weakly mixing and there is some countable dense subset $D$ of $X$ such that for any minimal system $(Y,S)$,
any point $y\in Y$ and any  open neighbourhood $V$ of $y$, and for any nonempty open subset $U\subset X$,
there is $x\in D\cap U$ satisfying that $\{n\in\Z_+: T^nx\in U, S^ny\in V\}$ is syndetic. Some characterization
for the general case is also described.

As applications we show that if a transitive system $(X,T)$ is disjoint from all minimal systems, then so are
$(X^n,T^{(n)})$ and $(X, T^n)$ for any $n\in \N$. It turns out that a transitive system $(X,T)$  is disjoint from all minimal
systems if and only if the hyperspace system $(K(X),T_K)$  is disjoint from all minimal systems.

\end{abstract}

\maketitle





\section{Introduction}

The notion of {\em disjointness} of two dynamical systems, both in ergodic
theory and in topological dynamics, was introduced by Furstenberg in his
seminal paper \cite{Fur67}. This notion plays an important role for ergodic systems,
see for instance \cite{G}. Compared with ergodic theory, there still remain in
topological dynamics, some basic problems to settle. We refer to \cite{HY05,O10, O17,LYY15, LOYZ, D, DSY12, GW} for recent developments.

By a {\em topological dynamical system} (t.d.s.) we mean a pair $(X,T)$, where
X is a compact metric space (with metric $d$) and $T: X\rightarrow X$ is continuous
and surjective. Let $(X,T)$ and $(Y,S)$ be two t.d.s. We say $J\subset X\times Y$ is a {\em joining} of $X$ and $Y$ if $J$ is a nonempty, closed, invariant set, which is mapped
onto $X$ and $Y$ by the respective coordinate projections. The product $X\times Y$
is always a joining and when it is the only joining we say that $(X,T)$ and
$(Y,S)$ are {\em disjoint}, denoted by $(X,T)\perp (Y,S)$ or $X\perp Y$. Note that if $(X,T)\perp (Y,S)$ then one of them is minimal \cite{Fur67}, and if in addition $(Y,S)$ is minimal then the set of recurrent points of $(X,T)$ is dense in $X$ \cite{HY05}.

In \cite{Fur67}, Furstenberg showed that each totally transitive system
with dense set of periodic points is disjoint from any minimal
system; and each weakly mixing system is disjoint from any minimal
distal system. He left the following question:

\medskip
\noindent {\bf Problem} \cite[Problem G]{Fur67}: {\em Describe the system who is disjoint
from all minimal systems.}

\medskip

Let $\mathcal{T}$ be a class of t.d.s. and $(X,T)$ be a t.d.s. If
$(X,T)\perp (Y,S)$  for every $(Y,S)\in \mathcal{T}$, then we denote it
by $(X,T)\perp \mathcal{T}$, and let $\mathcal{T}^\perp=\{(X,T): (X,T)\perp \mathcal{T}\}$.
Let $\mathcal{M}$ be the class of all minimal systems. Hence
Furstenberg's problem can be restated as follows: {\em Describe the class ${\mathcal{M}}^\perp$.}

\medskip

In \cite{HY05}, it was shown that a transitive t.d.s. disjoint with all
minimal systems has to be a weakly mixing $M$-system; and each weakly mixing system
with dense small periodic sets\footnote{We say that $(X,T)$ has dense small periodic sets if for
any nonempty open subset $U$ of $X$ there exists a closed subset $Y$ of $U$ and $k\in \N$
such that $T^kY\subset Y$. Clearly, every transitive system with dense set of periodic points (so-called $P$-system) has dense small periodic sets.} is disjoint from any minimal system.
Then in \cite{DSY12,O10}, this result was generalized as follows: every weakly mixing system with dense distal points is disjoint
with any minimal system. A further effort was made in \cite{LYY15}: if $(X,T)$ is weakly mixing and $(K(X), T_K)$ has dense distal points,
then $(X,T)$ is disjoint with all minimal systems, where $(K(X), T_K)$ is the hyperspace system of $(X,T)$. In \cite{LOYZ} an example $(X,T)$
was constructed such that $(K(X), T_K)$ has dense distal points, and at the same time $(X,T)$ does not have dense distal points.

\medskip

Let $(X,T)$ be a t.d.s. Recently, Oprocha gave the following result \cite{O17}:  If $(X, T)$ is weakly mixing and for every minimal system $(Y, S)$
there exists a countable set $D\subseteq X$ such that for every nonempty open set $U$ of $X$ the following
condition holds:
\begin{enumerate}
  \item for any $y\in Y$ and any open neighbourhood $V$ of $ y$ there is $x\in D\cap U$ such that the set of
transfer times $N_{T\times S}((x,y),U\times V)=\{n\in\Z_+: T^nx\in U, S^ny\in V\}$ is syndetic,
\end{enumerate}
then $(X, T)$ is disjoint with every minimal system.

\medskip

Oprocha asked that whether assumptions in his theorem are also a necessary condition for
$M$-systems to be disjoint with all minimal systems?

\medskip









In this paper, we show the following theorem:

\begin{thm} \label{new-main}
Let $(X,T)$ be a transitive t.d.s. Then $(X,T)$ is disjoint from all minimal systems if and only if
$(X,T)$ is weakly mixing and there is some countable dense subset $D$ of $X$ (consisting of minimal points)
such that for any minimal system $(Y,S)$, any point $y\in Y$ and any  open neighbourhood $V$ of $y$,
and for any nonempty open subset $U\subset X$,
there is $x\in D\cap U$ satisfying that $N_{T\times S}((x,y),U\times V)$ is syndetic.
\end{thm}
Hence we give an answer to Furstenberg's problem, and also answers the question by Oprocha affirmatively in some sense.
Note that the countable set in our theorem is universal for any minimal system.

\medskip

Central sets were introduced by Furstenberg, and they have very rich combinatorial properties \cite{Fur81}. A subset $A$ of $\Z_{+}$ is called a {\it dynamical syndetic set}, if there exist a minimal system $(Y,S)$, $y\in Y$ and
an open neighbourhood $V_y$ of $y$ such that $A\supset N_S(y,V_y)=\{n\in \Z_+: S^ny\in V_y\}$. We will show (Theorem \ref{C-thm})
that a set $S$ is a central set if and only if $S=A\cap B$, where $A$ is thick and $B$
is dynamical syndetic. Using the notion of central sets we can give the following theorem.

\begin{thm}
Let $(X,T)$ be a transitive t.d.s. Then $(X,T)$ is disjoint from all minimal systems
if and only if $(X,T)$ is weakly mixing and there is a countable dense subset $D$ of $X$ such
that for each nonempty open subset $U$ of $X$ and each central set $S=A\cap B$, we can find $x=x(B)\in D\cap U$ (independent of $A$) such that $N_T(x,U)\cap S\neq \emptyset$,
where $A$ is
thick and $B$ is dynamical syndetic.
\end{thm}

It is clear that the above two theorems can be stated for a transitive system disjoint from a given
minimal system. The following theorem was proved in \cite[Corollary 4.6]{HY05}, which can be considered as an answer to Furstenberg's problem in the general case.

\begin{thm} \label{hunagye05} Let $(X,T)$ be a t.d.s. Then $(X,T)\perp \mathcal{M}$ if and only if for
any minimal system $(Y,S)$, there exist countably many transitive subsystems of $(X,T)$ such that their union is dense in
$X$ and each of them is disjoint from $Y$.
\end{thm}

As a direct application we have
\begin{prop}\label{C-set}
Let $(X,T)$ be a weakly mixing system. If for each nonempty open subset $U$ of $X$
there is $x\in U$ such that $N_T(x,U)$ is a $C^*$-set, then $(X,T)$ is disjoint from
all minimal systems, i.e. $(X,T)\in \mathcal{M}^\perp$.
\end{prop}

As other applications we show that if a transitive system $(X,T)$ is disjoint from all minimal systems, then so are $(X^n,T^{(n)})$ and $(X, T^n)$
for any $n\in \N$, where $T^{(n)}=T\times \ldots \times T $ ($n$ times). Combining this result with other results it turns out that a transitive system $(X,T)$  is disjoint from all minimal systems
if and only if the hyperspace system $(K(X),T_K)$  is disjoint from all minimal systems.

\begin{rem} We remark that in \cite{HY05} Huang and Ye gave a necessary and sufficient condition for a transitive system
being disjoint from all minimal systems through the notion of m-set (see Proposition \ref{hy-main}). The condition given in this paper
is easier to handle than that one.
\end{rem}

\medskip
The paper is organizing as follows. In Section 2 we will give some preliminaries and we prove our main results in Section 3.
Then in Section 4 we give some applications of our results.

\medskip
\noindent {\bf Acknowledgments.} The authors would like to thank P. Oprocha for sharing us the early version of his recent paper. We also thank the referee for the very careful reading and many useful comments, which help us to improve the writing of the paper.

\section{Preliminary}

In the article, integers, nonnegative integers and natural numbers
are denoted by $\Z$, $\Z_+$ and $\N$ respectively.

\subsection{Topological dynamical system}\
\medskip

By a {\em topological dynamical system} (t.d.s.) we mean a pair $(X,T)$, where
$X$ is a compact metric space (with metric $d$) and $T:X\to X$ is
continuous and surjective. A nonempty closed invariant subset $Y
\subset X$ defines naturally a {\em subsystem} $(Y,T)$ of $(X,T)$.

The {\em orbit} of $x$, $orb(x,T)$ (or simply $orb(x)$), is the set
$\{T^nx: n\in \Z_+\}$.
A t.d.s. $(X,T)$ is {\it transitive} if for each pair of nonempty open subsets $U$ and $V$, $$N_T(U,V)=\{n\in\Z_+: U\cap T^{-n}V\not=\emptyset\}$$
is infinite. Equivalently, $(X,T)$ is transitive if and only if
there exists $x\in X$ such that
$\overline{orb(x,T)}=X$; such $x$ is called a {\em transitive
point}, and the set of transitive points is denoted by $Tran_T$. It
is well known that if a system $(X,T)$ is transitive
then $Tran_T$ is a dense $G_\delta$ set. A system $(X,T)$ is {\it weakly
mixing} if $(X\times X, T\times T)$ is transitive.

A t.d.s. $(X,T)$ is {\em minimal} if $Tran_T=X$. Equivalently,
$(X,T)$ is minimal if and only if it contains no proper subsystems.
A point $x \in X $ is {\em minimal} or {\em almost periodic} if the
subsystem $(\overline{orb(x,T)},T)$ is minimal.

A t.d.s. $(X,T)$ is an $M$-{\it system} if it is transitive and the
set of minimal points is dense.

\medskip

Let $(X,T)$ be a t.d.s. and $(x,y)\in X^2$. It is a {\it proximal}
pair if there is a sequence $\{n_i\}$ in $\Z_+$ such that
$\lim_{i\to +\infty} T^{n_i} x =\lim_{i\to +\infty} T^{ n_i} y$; and
it is a {\it distal} pair if it is not proximal. Denote by $P(X,T)$
or $P_X$ the set of all proximal pairs of $(X,T)$. For a point $x\in X$, $P[x]=\{y\in X: (x,y)\in P(X,T)\}$ is called {\em proximal cell} of $x$.
A point $x$ is
said to be {\em distal} if whenever $y$ is in the orbit closure of
$x$ and $(x,y)$ is proximal, then $x = y$.

A t.d.s. $(X,T)$ is
called {\it distal} if $(x,x')$ is distal whenever $x,x'\in X$ are
distinct.
A t.d.s. $(X,T)$ is {\it equicontinuous} if for every $\ep>0$ there
exists $\d>0$ such that $d(x_1,x_2)< \d$ implies
$d(T^nx_1,T^nx_2)<\ep$ for every $n\in \Z_+$. As we assume that $T$ is surjective, it is easy to see that
each equicontinuous system is distal.

\medskip

For a t.d.s. $(X,T)$, $x\in X$ and $U\subset X$ let
$$N_T(x,U)=N(x,U)=\{n\in \Z_+: T^nx\in U\}.$$
A point $x\in X$ is said to be {\em recurrent} if for every neighborhood $U$ of $x$, $N_T(x,U)$ is infinite.


\subsection{Furstenberg families}\
\medskip

Let us recall some notions related to Furstenberg families (for
details see \cite{Fur81}). Let $\P=\P({\Z}_{+})$ be the collection
of all subsets of $\Z_+$. A subset $\F$ of $\P$ is a {\em
(Furstenberg) family}, if it is hereditary upwards, i.e. $F_1
\subset F_2$ and $F_1 \in \F$ imply $F_2 \in \F$. A family $\F$ is
{\it proper} if it is a proper subset of $\P$, i.e. neither empty
nor all of $\P$. It is easy to see that $\F$ is proper if and only
if ${\Z}_{+} \in \F$ and $\emptyset \notin \F$. Any subset
$\mathcal{A}$ of $\P$ can generate a family $[\mathcal{A}]=\{F \in
\P:F \supset A$ for some $A \in \mathcal{A}\}$. If a proper family
$\F$ is closed under intersection, then $\F$ is called a {\it
filter}. For a family $\F$, the {\it dual family} is
$$\F^*=\{F\in\P: {\Z}_{+} \setminus F\notin\F\}=\{F\in \P:F \cap F' \neq
\emptyset \ for \ all \ F' \in \F \}.$$ $\F^*$ is a family, proper
if $\F$ is. Clearly, $(\F^*)^*=\F$ and ${\F}_1\subset {\F}_2$ implies that ${\F}_2^* \subset {\F}_1^*.$
Denote by $\F_{inf}$ the family consisting of all infinite subsets of $\Z_+$.

\subsection{$\F$-recurrence and some important families}\
\medskip

Let $\F$ be a family and $(X,T)$ be a t.d.s. We say $x\in X$ is
$\F$-{\it recurrent} if for each neighborhood $U$ of $x$, $N_T(x,U)\in
\F$. So the usual recurrent point is just $\F_{inf}$-recurrent one.







A subset $S$ of $\Z_+$ is {\it syndetic} if it has a bounded gaps,
i.e. there is $N\in \N$ such that $\{i,i+1,\cdots,i+N\} \cap S \neq
\emptyset$ for every $i \in {\Z}_{+}$. $S$ is {\it thick} if it
contains arbitrarily long runs of positive integers, i.e. there is a
strictly increasing subsequence $\{n_i\}$ of $\Z_+$ such that
$S\supset \bigcup_{i=1}^\infty \{n_i, n_i+1, \ldots, n_i+i\}$. The
collection of all syndetic (resp. thick) subsets is denoted by
$\F_s$ (resp. $\F_t$). Note that $\F_s^*=\F_t$ and $\F_t^*=\F_s$.

A classic result stated that $x$ is a minimal point if and only if
$N_T(x,U)\in \F_s$ for any neighborhood $U$ of $x$ \cite{GH}. And a t.d.s.
$(X,T)$ is weakly mixing if and only if $N_T(U,V)\in \F_t$ for any
nonempty open subsets $U,V$ of $X$ \cite{Fur67, Fur81}.

A subset $S$ of $\Z_+$ is {\it piecewise syndetic} if it is an
intersection of a syndetic set with a thick set. Denote the set of
all piecewise syndetic sets by $\F_{ps}$. It is known that a t.d.s.
$(X,T)$ is an $M$-{\it system} if and only if there is a transitive
point $x$ such that $N_T(x,U)\in \F_{ps}$ for any neighborhood $U$ of
$x$ (see for example \cite[Lemma 2.1]{HY05}).

Let $\{ p_i \}_{i=1}^\infty$ be a sequence in
$\mathbb{N}$. One defines $$FS(\{ p_i \}_{i=1}^\infty)=
\Big\{\sum_{i\in \alpha} p_i: \alpha \text{ is a
nonempty finite subset of } \N\Big \}.$$ $F$ is an {\it IP set} if it
contains some $FS({\{p_i\}_{i=1}^{\infty}})$, where $p_i\in\N$.
The collection of all IP sets is denoted by $\F_{ip}$. A subset of
$\N$ is called an {\it ${\text{IP}}^*$-set}, if it has nonempty
intersection with any IP-set. It is known that a point $x$ is a
recurrent point if and only if $N_T(x,U)\in \F_{ip}$ for any
neighborhood $U$ of $x$, and $x$ is distal if and only if $x$ is
$IP^*$-recurrent \cite{Fur81}.

\subsection{Systems $(K(X),T_K)$ and $(M(X), T_M)$}\
\medskip

Let $(X,T)$ be a topological dynamics
and let $\B(X)$ be the Borel $\sigma$-algebra of $X$. Let $M(X)$ be the collection of all Borel probability measures on $X$ with the weak$^*$ topology. Then
$T$ induces a map $T_M$ on $M(X)$ naturally by sending $\mu\in M(X)$ to $T\mu$, where $T\mu$ is defined by $T\mu(A)=\mu(T^{-1}A)$ for all $A\in \B(X)$.

Let $K(X)$  be the space of all nonempty closed subsets of $X$ endowed with
Hausdorff metric, and $T_K : K(X)\rightarrow  K(X)$  be the  induced map defined by $T_K(A) =
T(A)=\{Tx: x\in A\}$ for any $A \in  K(X)$.

\subsection{$m$-sets}\
\medskip

The notion of $m$-set was introduced in \cite{HY05}.

\begin{de} A subset $A$ of $\Z_{+}$ is called an {\it
$m$-set}, if there exist a minimal system $(Y,S)$, $y\in Y$ and
a nonempty open subset $V$ of $Y$ such that $A\supset N_S(y,V)$.
\end{de}

Recall that a subset $A$ of $\Z_{+}$ is a dynamical syndetic set, if there exist a minimal system $(Y,S)$, $y\in Y$ and an open neighbourhood $V_y$ of $y$ such that $A\supset N_S(y,V_y)$. 
Let $A$ be a $m$-set. Then there exist a minimal system $(Y,S)$, $y\in Y$ and a nonempty open subset $V$ of $Y$ such that $A\supset N_S(y,V)$. Since $(Y,S)$ is minimal, there exists some $k$ such that $S^ky\in V$. Thus $S^{-k}V$ is an open neighbourhood of $y$ and $N_S(y, S^{-k}V)$ is a dynamical syndetic set. Note that $N_S(y,V)\supset N_S(y,S^{-k}V)+k$. It follows that each $m$-set is a translation of some dynamical syndetic set.

\medskip

For a transitive system whether it is in ${\mathcal M}^\perp$ can be checked
through $m$-sets as the following theorem shows.
For a minimal dynamical system $(Y,S)$, we define
\begin{align*}
&\mathcal{F}_Y=\{ A\subset \Z_+:A\supset N_S(y,V) \text{ for some } y\in
Y\text{ and nonempty open subset }\ V\}\ \text{and}\\
& \mathcal{F}^*_Y=\{ B\subset \Z_+:B\cap A \not=\emptyset \text{ for
each } A \in \mathcal{F}_Y \}.
\end{align*}

\begin{prop} \label{hy-main}
Let $(X,T)$ be a transitive system and
$x\in \text{Trans}_T$. Then

\item {(1)} $(X,T)\in {\mathcal M}^\perp$ if and only if $N_T(x,U)\cap A \not=\emptyset$
for any neighborhood $U$ of $x$ and any m-set $A$.

\item {(2)} $(X,T)\perp (Y,S)$ if and only if for any open neighborhood $U$
of $x$, one has $N_T(x,U)\in  \mathcal{F}^*_Y$.
\end{prop}

We say that $(X,T)$ is {\it strongly disjoint from all minimal systems} if $(X^n, T^{(n)})$
is disjoint from all minimal systems for any $n\in\N$, where $T^{(n)}=T\times \ldots \times T $ ($n$ times). Then we have

\begin{prop} \cite{LYY15}\label{JYY-thm}
Let $(X, T)$ be a t.d.s. Then
\begin{enumerate}
\item If $(K(X),T_K)$ is weakly mixing and is disjoint from all minimal systems, then $(X,T)$ is weakly mixing and is disjoint from all minimal systems.

\item  If $(X,T)$ is strongly disjoint from all minimal systems, then both $(K(X),T_K)$ and $(M(X),T_M)$ are disjoint from all minimal systems.
\end{enumerate}
\end{prop}

\section{Main results}

\subsection{Basic lemmas}\
\medskip

To show the main result we need the following two propositions. The first one was obtained in \cite{HY05} by using
Proposition \ref{hy-main}.

\begin{lem} \label{huang-ye-05}
If a transitive system is disjoint from all minimal systems, then it is weakly mixing and has a dense set of minimal points, i.e. it is a weakly mixing M-system.
\end{lem}

The following lemma was proved in \cite{AK}.

\begin{lem} \label{akin-k}
Let $(X,T)$ be a weakly mixing system. Then each proximal cell is residual, that is, for each $x\in X$, $P[x]$ is residual in $X$.
\end{lem}

\subsection{Main results}\
\medskip

Now we are ready to show the main results of the paper.

\begin{thm}\label{train2}
Let $(X,T)$ be a transitive t.d.s. Then the following statements are equivalent.
\begin{enumerate}
\item $(X,T)$ is disjoint from all minimal systems.

\item $(X,T)$ is weakly mixing and there is a countable dense subset $D$ of $X$ such that for any minimal system $(Y,S)$,
any point $y\in Y$ and any open neighbourhood $V$ of $y$, and for any nonempty open subset $U\subset X$, there is $x\in D\cap U$ satisfying that $N_{T\times S}((x,y),U\times V)$ is syndetic.

\item $(X,T)$ is weakly mixing and there is a countable dense subset $D$ of $X$ consisting of minimal points such that for
any minimal system $(Y,S)$, any point $y\in Y$ and any  open neighbourhood $V$ of $y$, and for any nonempty open subset $U\subset X$,
there is $x\in D\cap U$ satisfying that $N_{T\times S}((x,y),U\times V)$ is syndetic.

\end{enumerate}
\end{thm}

\begin{proof}
(3)$\Longrightarrow (2)$. It is clear.

\medskip

(2)$\Longrightarrow (1)$. We follow the proofs in \cite{DSY12, O17}.
By the weak mixing property of $(X,T)$ and countability of $D$, there is a point $x_0\in X$ such that $(x_0,x)$ is proximal for any $x \in D$ by Lemma \ref{akin-k}.

Assume now $J\subset X\times Y$ is a joining. Then there is $y\in Y$ such that $(x_0,y)\in J$.
Let $V$ be an open neighborhood of $y$ and $U$ be any nonempty open set of $X$. Take a nonempty
 open subset $U'$ of $X$ and $\ep>0$ such that $B_\ep(\overline{U'})\subseteq U$. Then by the assumption,
 there is $x\in D\cap U'$ such that $N_{T\times S}((x,y),U'\times V)$ is syndetic. At the same time, since $(x,x_{0})$ is proximal, we have that
$$\{k\in \Z_+: d(T^k(x),T^k(x_{0}))<\ep\}$$
is thick. This implies that there is $k\in\Z_+$ such that
$$S^k(y)\in  V,\ \ T^k(x) \in U',\ \  \text{and}\ d(T^k(x_0),T^k(x))<\ep.$$
So  $T^k(x_0)\in U$. We have $(T^k(x_0),S^k(y))\in J\cap {\overline U}\times \overline {V}$.
Since $U$ is an arbitrary open set of $X$ and $V$ is an arbitrary open neighborhood of $y$,
it implies that $X\times \{y\}\subset J$. Thus we get $J=X\times Y$ as $y$ has a dense orbit in $Y$.

\medskip
(1)$\Longrightarrow (3)$.
Suppose now that $(X,T)$ is disjoint from all minimal systems. By Proposition \ref{hy-main}, $(X,T)$ is an $M$-system.
Let $D$ be any countable dense subset of $X$ consisting of minimal points. Now we show it is what we need.

Assume the contrary that the condition in (3) does not hold for $D$. Then there are a minimal system $(Y,S)$, a point $y\in Y$ and its open neighbourhood $V$,
and there is a nonempty open set $U\subset X$ such that $N_{T\times S}((x,y),U\times V)$ is not syndetic for any point $x\in D\cap U$.

\medskip

\noindent {\em Claim:} For each $x\in D\cap U$, there is some $T\times S$-invariant subset $J_{(x,y)}\subseteq (U\times V)^c$ such that $x\in p_1(J_{(x,y)})$, where $p_1$ is the projection of $X\times Y$ to the first coordinate.


\begin{proof}[Proof of Claim]
Since $N_{T\times S}((x,y),U\times V)$ is not syndetic, $A=N_{T\times S}((x,y),(U\times V)^c)$ is a thick set.
As $x$ is a minimal point, for each $\ep>0$, $N_T(x,B_\ep(x))$ is a syndetic set. Thus for any fixed $\ep>0$ and $L\in \N$, there are infinitely many $n\in N_T(x,B_\ep(x))$ such that
$$n,n+1,\ldots,n+L\in A.$$
From this fact we choose an increasing sequence $\{n_k\}\subseteq \N$ such that $n_k\in N_T(x,B_{1/k}(x))$  and
$n_k,n_k+1,\ldots,n_k+k \in A.$ Without loss of
generality assume that $$\lim_{k\to \infty} (T\times S)^{n_k}(x,y)=(x_0,y_0).$$ Since $n_k\in N_T(x,B_{1/k}(x))$, it is clear that $x_0=x$. Now for each $m$, we have that
$$(T\times S)^m(x_0,y_0)=\lim_{k\to \infty} (T\times S)^{n_k+m}(x,y)\in (U\times V)^c.$$
Now let
$$J_{(x,y)}=\overline{orb((x_0,y_0),T\times S)}.$$ Then $J_{(x,y)}\subseteq (U\times V)^c$, and $x=x_0\in p_1(J_{(x,y)})$.
This ends the proof of the claim.
\end{proof}

Now let $$J=\overline{\bigcup_{x\in D\cap U}J_{(x,y)}}.$$
Then it is clear that $J$ is $T\times S$-invariant and $J\subseteq (U\times V)^c$. By Claim, we have that $D\cap U\subseteq p_1(J)$. Since $D$ is a dense set of $X$, it follows that $$U\subseteq p_1(J).$$
Since $(X,T)$ is transitive and $p_1(J)$ is $T$-invariant and closed, it deduces that $p_1(J)=X$. Moreover, as $(Y,S)$ is minimal, the projection
of $J$ to the second coordinate is $Y$. We conclude that
$J$ is a joining of $X$ and $Y$. But $J\subset (U\times V)^c$, a contradiction.
\end{proof}


We also have the following theorem which is easier to handle in some situations.

\begin{thm}\label{train}
Let $(X,T)$ be a transitive t.d.s. Then the following statements are equivalent.
\begin{enumerate}
\item $(X,T)$ is disjoint from all minimal systems.


\item $(X,T)$ is weakly mixing and for any nonempty open subset $U\subset X$ and any nonempty open subset $V\subset Y$ of
any minimal system $(Y,S)$, there is $x=x(U,V)\in U$  such that $N_{T\times S}((x,y),U\times V)$ is syndetic for any $y\in V$.

\item $(X,T)$ is weakly mixing and for any nonempty open subset $U\subset X$ and any nonempty open subset $V\subset Y$ of
any minimal system $(Y,S)$, there is a minimal point $x=x(U,V)\in U$  such that $N_{T\times S}((x,y),U\times V)$ is syndetic for any $y\in V$.

\end{enumerate}
\end{thm}

\begin{proof} The proof is only slightly different from that of Theorem \ref{train2}.
For the sake of completeness, we give the proof. Obviously, (3)$\Longrightarrow (2)$.

\medskip

(2)$\Longrightarrow (1)$. We follow the proofs in \cite{DSY12, O17}. Assume that $\{U_n\}_{n=1}^\infty$ and $\{V_m\}_{m=1}^\infty$ are bases for the topologies
of $X$ and $Y$ respectively. Then by (2) for given $U_n$ and $V_m$, there is $x_{n,m}\in U_n$ such that $N_{T\times S}((x_{n,m},y),U_n\times V_m)$ is syndetic for any $y\in V_m$.
By the weak mixing property of $(X,T)$, there is a point $x\in X$ such that $(x,x_{n,m})$ is proximal
for any $n,m\in\N$ by Lemma \ref{akin-k}.

Assume now $J\subset X\times Y$ is a joining. Then there is $y\in Y$ such that $(x,y)\in J$.
Let $V$ be an open neighborhood of $y$. Then there is $m_0\in \N$ such that $y\in V_{m_0}\subset V$. We will show that
$(x_{n,m_0},y)\in J$ for any $n\in\N$.

To do this, let $U$ be an open neighborhood of $x_{n,m_0}$. Then there is $\epsilon>0$ such that $B_{2\ep}(x_{n,m_0})\subset U$
and put $W=B_{\ep}(x_{n,m_0})$. Then there is $x_{n',m_0}\subset U_{n'}\subset W$.
Then we have  $N_{T\times S}((x_{n',m_0},y), U_{n'}\times V_{m_0})$ is syndetic. At the same time, since $(x,x_{n',m_0})$ is proximal, we have that
$$\{k\in \Z_+: d(T^k(x),T^k(x_{n',m_0}))<\ep\}$$
is thick. This implies that there is  $k\in\Z_+$ such that
$$S^k(y)\in V_{m_0}\subset V,\ \ T^k(x_{n',m_0}) \in U_{n'},\ \text{and}\ d(T^k(x),T^k(x_{n',m_0}))<\ep.$$
So $T^k(x)\in U$. We have $J\cap {\overline U}\times \overline {V}\not=\emptyset.$
It deduces that $(x_{n,m_0},y)\in J$ for any $n\in\N$.

It implies that $X\times \{y\}\subset J$
and hence we get $J=X\times Y$ as $y$ has a dense orbit.

\medskip
(1)$\Longrightarrow (3)$.
Suppose now that $(X,T)$ is disjoint from $(Y,S)$. By Lemma \ref{hy-main}, $(X,T)$ is an M-system. Assume the contrary that the condition does not hold.
Then there are a nonempty open set $U\subset X$, and a nonempty open set
$V\subset Y$ satisfying that for any minimal point $x\in U$ there
is $y=y(x)\in V$ such that $N_{T\times S}((x,y),U\times V)$ is not syndetic.

Let $x\in U$ be a transitive point. As $(X,T)$ is an $M$-system, we can choose minimal points $x_n\in U$  such that $\lim_{n\rightarrow\infty} x_n=x$. Then for each $n\in\N$
there is $y_n=y_n(x_n)\in V$ such that $N_{T\times S}((x_n,y_n),U\times V)$ is not syndetic.
That is, $N_{T\times S}((x_n,y_n),(U\times V)^c)$ is thick.

We now show that there is $(x_n',y_n')\in (U\times V)^c$ and $(x_n',y_n')\in \overline{orb((x_n,y_n), T\times S)}$
such that its orbit is outside $U\times V$. In fact there is a strictly increasing sequence $\{k_i\}_{i=1}^\infty$ of $\N$ such that
$(T\times S)^{k_i+j}(x_n,y_n)\in (U\times V)^c$ for each $i\in \N$ and $1\le j\le i$. Without loss of
generality assume that $\lim_{i\rightarrow\infty} (T\times S)^{k_i}(x_n,y_n)=(x_n',y_n')$. Then $(x_n',y_n')$
is the point we want. Let
$$J=\overline{\bigcup_{n=1}^\infty orb((x_n',y_n'),T\times S)}.$$
It is clear that $J$ is closed and $T\times S$-invariant. Moreover, as $(Y,S)$ is minimal, the projection
of $J$ to the second coordinate is $Y$. We now show that $p_1(J)=X$, where $p_1$ is the projection of $J$ to the first coordinate.
We note that $p_1(J)$ is $T$-invariant and closed.

To do so, fix $n\in \N$. It is clear that there is a sequence $\{k_i\}$ such that $\lim_{i\rightarrow\infty} T^{k_i}x_n=x_n'$
by the construction of $x_n'$. Choose $0<\ep_n<1/n$.
Since $x_n$ is a minimal
point, $N_T(x_n,B_{\ep_n}(x_n))$ is syndetic. Assume that that $l_n$ is the gap of this subset of $\N$. By the continuity of $T$,
there is $\delta_n>0$ such that if $z_1,z_2\in X$ and $d(z_1,z_2)<\delta_n$ then $d(T^i(z_1),T^i(z_2))<\ep_n$ for each $i=1,\ldots,l_n$.
Since $\lim_{i\rightarrow\infty} T^{k_i}x_n=x_n'$ there is $j\in\N$ such that $d(T^{k_j}x_n,x_n')<\delta_n$. This implies that
 $$d(T^{k_j+i}x_n,T^ix_n')<\ep_n,\ i=1,\ldots,\ell_n.$$
There is $1\le i_n\le l_n$ such that $T^{k_j+i_n}x_n\in B_{\ep_n}(x_n)$.
This implies that $d(T^{i_n}(x_n'),x_n)<2\ep_n$ for each $n\in \N$.

Hence $x=\lim_{n\rightarrow \infty} x_n=\lim_{n\rightarrow \infty}  T^{i_n}(x_n')\in p_1(J)$, since $x_n'\in p_1(J)$,
and $p_1(J)$ is $T$-invariant and closed. It deduces that $p_1(J)=X$, as $x$ is a transitive point. We conclude that
$J$ is a joining of $X$ and $Y$. It is clear $J\subset (U\times V)^c$, a contradiction.
\end{proof}

By the same proof we have
\begin{thm}\label{train-y}
Let $(X,T)$ be a weakli mixing M-system and $(Y,S)$ is a minimal t.d.s. Then the following statements are equivalent.
\begin{enumerate}
\item $(X,T)\perp (Y,S)$.


\item For any nonempty open subset $U\subset X$ and any nonempty open subset $V\subset Y$,
 there is $x=x(U,V)\in U$  such that $N_{T\times S}((x,y),U\times V)$ is syndetic for any $y\in V$.

\item For any nonempty open subset $U\subset X$ and any nonempty open subset $V\subset Y$, there is a minimal point $x=x(U,V)\in U$  such that $N_{T\times S}((x,y),U\times V)$ is syndetic for any $y\in V$.

\end{enumerate}
\end{thm}

\subsection{Central sets and another form of the main result}\
\medskip

First we recall the following form of Auslander-Ellis theorem (see \cite[Theorem 8.7.]{Fur81} for example).

\begin{thm}[Auslander-Ellis]\label{thmAuslander}
Let $(X,T)$ be a compact metric t.d.s.. Then for any $x\in X$ and any $T$-invariant closed subset $Z$ of
$\overline {orb(x, T)}$, there is some minimal point $x' \in Z$ such that
$(x,x')$ is proximal.
\end{thm}

Using Auslander-Ellis's Theorem Furstenberg introduced a
notion called central set.  A subset $S\subseteq \Z_+$ is a {\em
central set} if there exists a t.d.s. $(X,T)$, a point $x\in X$ and
a minimal point $y$ proximal to $x$, and a neighborhood $U_y$ of $y$
such that $N_T(x,U_y)\subset S$. It is known that any central set is
an IP-set \cite[Proposition 8.10.]{Fur81}.
Denote the set of all central sets by $\F_C$. A set from $\F_C^*$ is called a $C^*$-set.
Note that all $IP^*$-sets are $C^*$-sets, but there is some $C^*$-set which is not $IP^*$.
Moreover, we know that a point of a dynamical system is $C^*$-recurrent if and only if it is $IP^*$-recurrent
if and only if it is distal \cite[Proposition 9.17]{Fur81}.

\medskip
Recall that a subset $A$ of $\Z_{+}$ is called a {\it dynamical syndetic set}, if there exist a minimal system $(Y,S)$, $y\in Y$ and an open neighbourhood $V_y$ of $y$ such that $A\supset N_S(y,V_y)$. Denote the set of all dynamical syndetic sets by $\F_{ds}$.
Let
\begin{equation*}
  \F_{dps}=\F_t\cap \F_{ds}=\{A\cap B: A\in F_t, B\in \F_{ds}\}.
\end{equation*}
Each element of $\F_{dps}$ is called a {\em dynamical piecewise syndetic set}.

\begin{thm}\label{C-thm}
$\F_{dps}=\F_{C}$.
\end{thm}

\begin{proof}
First we show that $\F_{C}\subseteq \F_{dps}$. Let $Q \in \F_C$. Then by definition, there
there exists a system $(X,T)$, a point $x\in X$ and
a minimal point $y$ proximal to $x$, and a neighborhood $U_y$ of $y$
such that $N_T(x,U_y)\subset Q$. Without loss of generality, we may assume that $U_y=B_{2\ep}(y)=\{z: d(z,y)<2 \ep\}$ for some $\ep>0$.
Let $$A=\{n\in\Z_+: d(T^nx,T^ny)<\ep\}.$$
Then $A$ is a thick set since $(x,y)$ is proximal \cite[Lemma 8.1.]{Fur81}. It is easy to verify that
$$A\cap N_T(y,B_\ep(y))\subseteq N_T(x,B_{2\ep}(y))\subseteq Q.$$
Hence $\F_{C}\subseteq \F_{dps}$.

\medskip

Now we show the converse. Let $Q\in \F_{dps}$. Then there is a thick set $A$ and a dynamical syndetic
set $B$ such that $Q = A\cap B$. Let $(Y,S)$ be a minimal system, $y\in Y$ and an open neighbourhood $V_y$ of $y$ such that $B\supset N_S(y,V_y)$.

Let $(\Sigma_2=\{0,1\}^\Z_+,\sigma)$ be the shift system. Let $\widetilde{X}=\Sigma_2\times Y$ and $\widetilde{T}=\sigma \times S$.
Let $x_0=(1_A, y)\in \widetilde{X}$, $X=\overline{orb(x_0,\widetilde{T})} $, and  $T=\widetilde{T}|_X$. Then $(X,T)$ is a t.d.s. \footnote{If $T$ is not surjective, one may embed $(X,T)$ into some surjective system. Let $Z=X\times D$, where $D=\{\frac 1n\}_{n\in \N}\cup \{0\}$. Define $R: Z\rightarrow Z$ satisfying $R(x,\frac{1}{n+1})=(x,\frac{1}{n})$, $n\in \N$; $R(x,1)=(Tx,1)$ and $R(x,0)=(x,0)$ for all $x\in X$. Then $(Z,R)$ is a t.d.s. and $R$ is surjective. Identifying $x$ with $(x,1)$ for all $x\in X$, $X$ can be viewed as a closed subset of $Z$ and $T=R|_X$. We cite this approach from the proof of \cite[Lemma 3.13]{FH12}.}

Now we show that $y_0=(1, y)\in X$ is a minimal point which is proximal to $x_0$.
First we show $y_0\in X$. Since $A$ is thick, there is an increasing sequence $\{n_i\}_{i=1}^\infty$ such that
$\{n_i,n_i+1,\ldots,n_i+i\}\subseteq A$ for all $i\in \N$. Without loss of generality, we assume that
$$x'=\lim_{i\to\infty} T^{n_i} x_0\in X.$$
By the construction of $\{n_i\}_{i=1}^\infty$, $x'$ has the form $(1,y')$ for some $y'\in Y$. As $(Y,S)$ is minimal,
there is some sequence $\{m_i\}_{i=1}^\infty$ such that $y=\lim_{i\to\infty} S^{m_i}y'$. Thus
$$y_0=(1,y)=\lim_{i\to\infty} T^{m_i}x' \in X.$$
Since $1$ is a fixed point of $\sigma$ and $y$ is minimal point of $(Y,S)$, $y_0$ is a minimal point of $(X,T)$.

Note that $\{n_i,n_i+1,\ldots,n_i+i\}\subseteq A$ for all $i\in \N$, and it follows that
$$\lim_{i\to \infty} d_X(T^{n_i}x_0, T^{n_i}y_0)=0,$$ where $d_X$ is the metric of $X$.
That is, $x_0$ and $y_0$ are proximal.

Let $[1]=\{\xi\in\Sigma_2 : \xi_0=1\}$. Then $V= ([1]\times V_y)\cap X$ is an open neighbourhood of $y_0$ in $X$.
Thus
$$Q=A\cap B\supseteq  A\cap N_S(y,V_y)=N_T\left( x_0, V \right)$$ is a central set, i.e. $Q\in \F_{C}$. Hence $\F_{dps}\subseteq \F_C$. The proof is completed.
\end{proof}

To show the another form of the main result we need the following lemma.

\begin{lem}\label{CC}
Let $(X,T)$ be a transitive t.d.s. If $(X,T)$ is disjoint from all minimal systems, then $(X,T)$ is weakly mixing and
there is a countable dense subset $D$ of $X$ such that for each nonempty open subset $U$ of $X$ and each central
set $S=A\cap B$ (where $A$ is thick and $B$ is dynamical syndetic),
we can find $x=x(B)\in D\cap U$ independent of $A$ such that $N_T(x,U)\cap S\neq \emptyset$.

\end{lem}

\begin{proof}
Assume that $(X,T)$ is disjoint from all minimal systems. By Lemma \ref{huang-ye-05},
$(X,T)$ is a weakly mixing M-system.  Then by Theorem \ref{train2}
 there is a countable dense subset $D$ of $X$ satisfying the condition in Theorem \ref{train2} (2).

Let $U$ be any nonempty open set of $X$. By Theorem \ref{C-thm} let $S=A\cap B$ be a central set,
where $A$ is thick and $B$ is a dynamical syndetic set. Let $(Y,S)$ be a minimal system, $y\in Y$ and
a nonempty open neighbourhood $V$ of $Y$ such that $B\supseteq N_S(y,V)$. For $y$ and $V$,
by the choice of $D$ there is some $x=x(B)\in D\cap U$ such that $N_{T\times S}((x,y),U\times V)$ is syndetic. Hence
\begin{equation*}
  \begin{split}
N_T(x,U)\cap S & =N_T(x,U)\cap A\cap B \supseteq N_T(x,U) \cap (A\cap N_S(y,V) )\\ & = A\cap N_{T\times S}((x,y),U\times V)\neq \emptyset.
\end{split}
\end{equation*}
The proof is completed.
\end{proof}

Now we are ready to give another form of the main result.
\begin{thm}\label{thm-central}
Let $(X,T)$ be a transitive t.d.s. Then $(X,T)\perp \mathcal{M}$
if and only if $(X,T)$ is weakly mixing and there is a countable dense subset $D$ of $X$ such
that for each nonempty open subset $U$ of $X$ and each central set $S=A\cap B$, we can find $x=x(B)\in D\cap U$  with $N_T(x,U)\cap S\neq \emptyset$, where $A$ is
thick and $B$ is dynamical syndetic.
\end{thm}

\begin{proof}
Assume that $(X,T)$ is weakly mixing and  there is a countable dense subset $D$ of $X$ such that for each
nonempty open subset $U$ of $X$ and each central set $S=A\cap B$ (where $A$ is thick and $B$ is dynamical syndetic),
we can find $x=x(B)\in D\cap U$ independent of $A$ such that $N_T(x,U)\cap S\neq \emptyset$. We show that $(X,T)$ is disjoint from all minimal systems.

Let $(Y,S)$ be a minimal system. We will show that for any nonempty open subset $U\subset X$ and any nonempty open subset
$V\subset Y$, $y\in V$, there is $x\in D\cap U$  such that $N_{T\times S}((x,y),U\times V)$ is syndetic. And hence
by Theorem \ref{train2}, $(X,T)$ is disjoint from all minimal systems.

Let $A\in \F_t$. Then by Theorem \ref{C-thm}, $A\cap N_S(y,V)\in \F_{dps}=\F_{C}$. By assumption, there is some $x\in D\cap U$ independent of $A$ such that
$$N_T(x,U) \cap (A\cap N_S(y,V) )\neq \emptyset.$$
That is,
$$A\cap N_{T\times S}((x,y),U\times V)=N_T(x,U) \cap (A\cap N_S(y,V) )\neq \emptyset.$$
As $A$ is an arbitrary thick set, $N_{T\times S}((x,y),U\times V)$ is syndetic.

\medskip

The converse follows from the proof of Lemma \ref{CC}. The proof is completed.
\end{proof}



\subsection{The general case}\
\medskip

Now we discuss Furstenberg's problem without the transitivity assumption. It was proved in \cite[Theorem 4.3]{HY05}
that if $(X,T)\perp \mathcal{M}$, then the set of minimal points of $(X,T)$ is dense in $X$.
Moreover, the following proposition was proved in \cite[Corollary 4.6]{HY05}.

\begin{prop} \label{hunagye05} Let $(X,T)$ be a t.d.s. Then $(X,T)\perp \mathcal{M}$ if and only if for
any minimal system $(Y,S)$, there exist countably many transitive subsystems of $(X,T)$ such that their union is dense in
$X$ and each of them is disjoint from $Y$.
\end{prop}

It is natural to conjecture the following: {\em If $(X,T)\perp \mathcal M$, then there are
countably many transitive subsystems of $(X,T)$ such that their union is dense in $X$ and each
of them is in ${\mathcal M}^\perp$.}
We remark that this conjecture is not true, since there is a distal system disjoint from all
minimal systems, see \cite[Example 4.10]{HY05}.

Together with Theorem \ref{train-y} and Proposition \ref{hunagye05} we get a description of a
dynamical system disjoint with all minimal systems. Since the characterization is not easy to handle,
it is a natural question to get some other intrinsic characterizations.

\section{Applications}

In this section we give several applications of the main theorem.

\subsection{Some sufficient conditions}\
\medskip

We say that a t.d.s. $(X,T)$ has {\em dense distal sets} if for each nonempty open subset
$U$ of $X$, there is a distal point $C$ of $(K(X),T_K)$ such that $C
\subset U$. It is shown that a system $(X,T)$ is a weakly mixing system with dense distal sets if and only if
$(K(X),T_K)$ is a weakly mixing system with dense distal points \cite{LYY15}.

Recall that a point is distal if and only if for any neighbourhood $U$ of $x$ and any open neighborhood $V$ of a minimal system $(Y,S)$, $N_{T\times S}\left((x,y),U \times V\right)$ is syndetic for any $y\in V$ if and only if $x$ is $IP^*$-recurrent \cite[Theorem 9.11.]{Fur81}. Hence by Theorem \ref{train}, we have the following corollary easily.

\begin{cor}\label{cor-result-1} The following classes are subset of $\mathcal{M}^\perp$.

\begin{enumerate}
\item $(X,T)$ is weakly mixing with dense distal sets.

\item $(X,T)$ is weakly mixing with dense distal points.

\end{enumerate}
\end{cor}

\begin{cor}\label{cor-result-2}
If $(X,T)$ is a weakly mixing t.d.s. such that for any minimal system $(Y,S)$ and any nonempty open set $U$ of $X$,  there are $x\in U$ and
a nonempty open set $V$ of $Y$ such that $(x,y)$ is minimal for any $y\in V$, then $(X,T)\in \mathcal{M}^\perp$.

In particular, if for any minimal system $(Y,S)$ and each nonempty open subset $U$ of $X$,
there is a minimal subset $M$ of $(X,T)$ which is disjoint from $Y$ and $U\cap M\not=\emptyset$, then $(X,T)\in \mathcal{M}^\perp$.
\end{cor}
\begin{proof} We assume that for any nonempty open set $U$ of $X$, there are $x\in U$ and a
nonempty open set $V$ of $Y$ such that $(x,y)$ is minimal for any $y\in V$. Since $(Y,S)$ is minimal, there is $n\in\N$ such that $\cup_{i=1}^n S^nV=Y$.
We note that if $(x,y)$ is minimal then $(x,Sy)$ is minimal. This follows from the fact that $\id\times S: (X\times Y,T\times S)\rightarrow (X\times Y,T\times S)$
is a factor map. Thus, $(x,y)$ is minimal for any $y\in Y$. Applying Theorem \ref{train}, we get the proof of the first statement.

Note that when the assumption of the second statement holds, as $M$ is disjoint from $Y$, we know that
$(x,y)$ is a minimal point of $T\times S$ for any $x\in U\cap M$ and any $y\in Y$.
\end{proof}

We remark that in fact Oprocha gave a very nice criteria which covers Corollary \ref{cor-result-1} and \ref{cor-result-2}  above. 

\begin{prop}\cite{O17} \label{1717}
Let $(X,T)$ be a weakly mixing system. If for each nonempty open subset $U$ of $X$ there is $x\in U$ such that
$N_T(x,U)$ is an  $IP^*$-set, then $(X,T)$ is
disjoint from all minimal systems, i.e. $(X,T)\in \mathcal{M}^\perp$.
\end{prop}

Using results in the previous section we have the following generalization.

\begin{prop}\label{C-set}
Let $(X,T)$ be a weakly mixing system. If for each nonempty open subset $U$ of $X$
there is $x\in U$ such that $N_T(x,U)$ is a $C^*$-set, then $(X,T)$ is disjoint from
all minimal systems, i.e. $(X,T)\in \mathcal{M}^\perp$.
\end{prop}
\begin{proof} It is a direct application of Theorem \ref{thm-central}.
\end{proof}

We note that the condition in Proposition \ref{1717} is not necessary as the example in \cite{O17} shows.
Now we show that in fact the same example in \cite{O17} indicates that condition in Proposition \ref{C-set} is also not necessary. One can verify it following the proof of Theorem 1.4 of \cite{O17}. For the sake of completeness, we give a slightly different proof.
To do so we need a lemma and a proposition.

\begin{lem} \label{wm-ye}Let $(X,T)$ be a t.d.s., $x\in X$ and $M$ be a minimal weakly mixing subsystem. If $x$ is proximal to some point in $M$
then $P[x]\cap M$ is dense in $M$.
\end{lem}
\begin{proof} We follow the argument in \cite{AK}. Assume that $x$ is proximal to a point $z\in M$. Let $G_k$ be the open balls (of $X$) of radius $1/k$ centred at $z$.
Let $U$ be a nonempty open subset of $M$.

Set $U_0=U$ and define inductively open sets $U_1,U_2,\ldots$ of $M$ and positive integers $n_k$ as follows.

Since $N_T(x, G_k)\in \mathcal{F}_{ps}$, and $N_T(U_{k-1},G_k\cap M)\in \mathcal{F}_{ts}$ (\cite[Theorem 4.7]{HY02}) we have
$$N_T(x, G_k)\cap N_T(U_{k-1},G_k\cap M)$$ is infinite.
So we can choose a nonempty open set $U_k$ of $M$ with closure contained in $U_{k-1}$ and an integer $n_k>k$ such that
$T^{n_k}x\in G_k$ and $T^{n_k}\overline{U_k}\subset G_k$. If $y$ is a point of the nonempty intersection
$\cap_k \overline{U}_k=\cap_k U_k$ then $T^{n_k}x\in G_k$ and $T^{n_k}(y)\in G_k$ and so $d(T^{n_k}x,T^{n_k}y))\le 2/k$. Thus $y$ is in $U$ and $x,y$ are proximal.
\end{proof}

\begin{prop} Let $(X,T)$ be a t.d.s. such that there are countably many non-trivial minimal subsystems $M_i$
which are weakly mixing, $\cup_{i=1}^\infty M_i$ is dense in $X$ and $\cup_{i=1}^\infty M_i$ is the set of minimal points of $X$.
Then there is a nonempty open subset $U$ of $X$ such that for any $z\in U$, $N_T(z,U)$ is not a $C^*$-set.
\end{prop}

\begin{proof} First there is a nonempty open subset $U$ such that
for each $i\in \N$, $U\cap M_i\not=\emptyset$ implies that ${\rm int}(U^c)\cap M_i\not=\emptyset$ (otherwise $T$ has a fixed point), where ${\rm int}(A)$ is the interior of a subset $A$.
It is clear that $U\not=X$ and we may assume that ${\rm int}(U^c)\not=\emptyset$. We will show that for any $z\in U$, $N_T(z,U)$ is not a $C^*$-set.


(1) If there is $i\in\N$ such that $z\in M_i$ is a minimal point, then $z$ is proximal to
a point $y\in V$, where $y\in M_i$, $V$ is a nonempty open subset such that $U\cap V=\emptyset$ (since $M_i$ is weakly mixing and ${\rm int} (U^c)\cap M_i\not=\emptyset$).

(2) If $z$ is not a minimal point, then $z$ is proximal to a minimal point $y_1\in M_i$ for some $i\in\N$. By Lemma \ref{wm-ye} $z$ is proximal to a minimal point $y\in V$,
where $V$ is a nonempty open subset of $X$ such that $U\cap V=\emptyset$  (as ${\rm int} (U^c)\cap M_i\not=\emptyset$).

In the above two cases we have $N_T(z,V)$ is a $C$-set. This implies that $N_T(z,U)$ is not a $C^*$-set since  $U\cap V=\emptyset$.
\end{proof}

\subsection{$(X^n,T^{(n)})$ and $(X,T^n)$}\
\medskip

It is known that if $(X,T)$ is weakly mixing, then so are $(X^n,T^{(n)})$ and $(X,T^n)$ for any $n\in\N$.
Now we show

\begin{thm} \label{product}
Assume that a transitive system $(X,T)$ is disjoint  from all minimal systems.
Then $(X^n,T^{(n)})$  is also disjoint from  all minimal systems for any $n\in \N$, i.e.
 $(X,T)$ is strongly disjoint from all minimal systems.
\end{thm}
\begin{proof}
Let $(Y,S)$ be a minimal system. Assume that $W$ is a nonempty open subset of $X^n$ and
$V$ is a nonempty open subset of $Y$. We may assume that $W\supset W_1\times \ldots
\times W_n$, where $W_i$ is a nonempty open subset of $X$. By the transitivity of $(X,T)$, there is a nonempty open
subset $U$ of $X$ such that for each $1\le i\le n$ there is $k_i\in\N$ with $T^{k_i}U\subset W_i$.
By Theorem \ref{train} there exists $x\in U$ such that $N_{T\times S}((x,y),U\times V)$ is syndetic for any $y\in V$.
This implies that
$$N_{T^{(n)}\times S}((T^{k_1}x,\ldots,T^{k_n}x),y), (T^{k_1}U\times \ldots \times T^{k_n}U)\times V)
\supset N_{T\times S}((x,y),U\times V)$$
is syndetic for any $y\in Y$.

Observing that $x'=(T^{k_1}x,\ldots,T^{k_n}x)\subset T^{k_1}U\times\ldots\times T^{k_n}U\subset W_1\times \ldots\times W_n
\subset W$ we get that $N_{T^{(n)}\times S}((x',y), W\times V)$ is syndetic for any $y\in Y$. Again applying Theorem \ref{train}
we get the conclusion.
\end{proof}

As a corollary we have

\begin{cor} Let $(X,T)$ be a transitive t.d.s. $(X,T)$  is disjoint from all minimal systems if and only if
$(K(X),T_K)$  is disjoint from all minimal systems.
\end{cor}
\begin{proof} It follows from Theorem \ref{product} and Proposition \ref{JYY-thm}.
\end{proof}

\begin{thm} Assume that a transitive system $(X,T)$ is disjoint from all minimal systems.
Then so is $(X,T^n)$  for any $n\in \N$.
\end{thm}
\begin{proof} Fix $n\in\N$ with $n\ge 2$ and let $(Y,S)$ be a minimal system. Set $\tilde{Y}=\cup_{i=1}^nY\times \{i\}$
and define $\tilde{S}: \tilde{Y}\rightarrow \tilde{Y}$ such that for any $y\in Y$,
$$\tilde{S}(y,i)=(y,i+1)\ \text{for}\ i=1,\ldots,n-1\ \text{and}\ \tilde{S}(y,n)=(Sy,1).$$
It is clear that $(\tilde{Y},\tilde{S})$ is also minimal and $\tilde{S}^{nk}(y,1)=(S^ky,1)$ for any $k\in\N$.

Let $U,V$ be open nonempty subsets of $X$ and $Y$ respectively. Then by Theorem \ref{train} we know that
there is $x\in U$ such that $N_{T\times \tilde{S}}((x, (y,1)), U\times (V\times \{1\}))$ is syndetic
for any $y\in Y$.
Assume that $k\in N_{T\times \tilde{S}}((x, (y,1)), U\times (V\times \{1\}))$ then there is some $k_1$ with $k=nk_1$
such that $(T^n)^{k_1}x\in U$ and ${S}^{k_1}(y)\in V$. So, $k_1\in N_{T^n\times S}((x,y),U\times V)$
which implies that $N_{T^n\times S}((x,y),U\times V)$ is also syndetic for any $y\in Y$. Again applying
Theorem \ref{train}, we conclude that $(X,T^n)$ is disjoint from all minimal systems.
\end{proof}

Finally, we restate a question in \cite{DSY12}

\begin{ques} Let $(X_1,T_1)$ and $(X_2,T_2)$ be transitive and be disjoint from all minimal systems.
Is it true that $(X_1\times X_2, T_1\times T_2)$ is also disjoint from all minimal systems?
\end{ques}
We note that we do not know the answer even for the very simple case when $X_1=X_2=X$, $T_1=T$ and $T_2=T^2$.

\end{document}